\documentclass[12pt, A4paper, oneside]{article}
\usepackage{amsthm,amsmath,amssymb,amscd,verbatim,epsfig,verbatim}
\usepackage{enumitem,multirow}
\usepackage{graphicx}
\usepackage{psfrag}
\theoremstyle{plain}
\newtheorem{thm}{Theorem}[section]
\newtheorem{lem}{Lemma}[section]

\newtheorem{rem}{Remark}[section]

\numberwithin{equation}{section}

\begin{document}

\fontsize{12}{25pt}\selectfont

\title{Optimal selection of  the $k$-th best candidate}

\author{
   Yi-Shen Lin\thanks{Institute of Statistical Science, Academia Sinica, Taipei 115, Taiwan, R.O.C. Email address: yslin@stat.sinica.edu.tw},\;\; Shoou-Ren Hsiau\thanks{Department of Mathematics, National Changhua University of Education, No. 1, Jin-De Rd., Changhua 500, Taiwan, R.O.C. Email address: srhsiau@cc.ncue.edu.tw},\; and Yi-Ching Yao\thanks{Institute of Statistical Science, Academia Sinica, Taipei 115, Taiwan, R.O.C. Email address: yao@stat.sinica.edu.tw}
    \date{}
    } \maketitle
\centerline {December 4, 2017}
\begin{abstract}
\fontsize{12}{18pt}\selectfont
In the subject of optimal stopping, the classical secretary problem is concerned with optimally selecting the best of $n$ candidates
when their relative ranks are observed sequentially. This problem has been extended to optimally selecting the $k$-th best candidate
for  $k\ge 2$. While the optimal stopping rule for $k=1,2$ (and all $n\ge 2$) is known to be of threshold type (involving one threshold),
we solve the case $k=3$ (and  all $n\ge 3$) by deriving an explicit optimal stopping rule that involves two thresholds. We also prove
several inequalities for $p(k,n)$, the maximum probability of selecting the $k$-th best of $n$ candidates. It is shown
that (i) $p(1,n)=p(n,n)>p(k,n)$ for $1<k<n$, (ii) $p(k,n)\ge p(k,n+1)$, (iii) $p(k,n)\ge p(k+1,n+1)$,
and (iv) $p(k,\infty):=\lim_{n\to \infty} p(k,n)$ is decreasing in $k$.
\end{abstract}
\begin{tabbing}
\>Secondary 62L15
\end{tabbing}

\fontsize{12}{20pt}\selectfont
\section{Introduction}
\hspace*{18pt}The classical secretary problem (also known as the best choice problem) has been extensively studied in the literature on optimal stopping, which is usually described as follows. There are $n$ (fixed) candidates to be interviewed sequentially in  random order for one secretarial position. It is assumed that these candidates can be ranked linearly without ties by a manager (rank 1 being the best). Upon interviewing a candidate, the manager is only able to observe the candidate's (relative) rank among those that have been interviewed so far.
The manager then  must decide whether to accept the present candidate (and stop interviewing) or to reject the candidate (and continue
 interviewing). No recall is allowed. The object is to maximize the probability of selecting the best candidate. More precisely, let $R_j$, $j=1,2,\ldots,n$, be the absolute rank of the $j$-th candidate
such that $(R_1,\dots,R_n)=\sigma_n$ with probability $1/n!$ for every permutation $\sigma_n$ of $(1,2,\dots,n)$.
Define $X_j=|\{1\le i \le j: R_i \le R_j\}|$, the relative rank of the $j$-th candidate among the first $j$ candidates.
It is desired to find a stopping rule $\tau_{1,n}\in \mathcal{M}_n$ such that
$P(R_{\tau_{1,n}}=1)=\sup_{\tau\in \mathcal{M}_n}P(R_{\tau}=1)$
where $\mathcal{M}_n$ denotes the set of all  stopping rules adapted to the filtration $\{\mathcal{F}_j\}$,
$\mathcal{F}_j$ being  the $\sigma$-algebra generated by $X_1,X_2,\dots,X_j$. It is well known ({\it cf.} Lindley \cite{ref10}) that the optimal stopping rule $\tau_{1,n}$ is of threshold type given by
$\tau_{1,n}=\min\{r_n\le j \le n: X_j=1\}$ where  $\min \emptyset:=n$ and the threshold
 $r_n:=\min\{j\ge1:\;\sum_{i=j+1}^n\frac{1}{i-1}\le1\}$.
 Moreover, the maximum probability of selecting the best candidate (under $\tau_{1,n}$) is
$p(1,n):=\frac{r_n-1}{n}\sum_{i=r_n}^n\frac{1}{i-1}$,
which converges as $n \to \infty$ to $p(1,\infty):=1/e=\lim_{n \to \infty} r_n/n$.

A great many interesting variants of the secretary problem have been formulated and solved in the literature
({\it cf.} the review papers by Ferguson \cite{ref4} and Freeman \cite{ref6} and  Samuels \cite{ref19}), most of which are concerned with optimally
selecting the best candidate or one of the $k$ best candidates. In contrast, only a few papers ({\it cf.} Rose \cite{ref16}, Szajowski \cite{ref26}
and Vanderbei \cite{ref21}) considered and solved the problem of optimally selecting the second best candidate.
(According to Vanderbei \cite{ref21}, in 1980, E.B. Dynkin proposed this problem to him with the motivating story that ``We are
trying to hire a postdoc and we are confident that the best candidate will receive and accept an offer from Harvard.'' Thus Vanderbei \cite{ref21}
refers to the problem as the postdoc variant of the secretary problem.)
These authors showed that the optimal stopping rule
$\tau_{2,n}$ is also of threshold type given by $\tau_{2,n}=\min\{r_n' \le j \le n: X_j=2\}$ with $r_n'=\lceil\frac{n+1}{2}\rceil$
(the smallest integer not less than $\frac{n+1}{2}$), which attains
the maximum probability of selecting the second best candidate
$$p(2,n):=P(R_{\tau_{2,n}}=2)=\sup_{\tau \in \mathcal{M}_n} P(R_{\tau}=2)=\frac{(r_n'-1)(n-r_n'+1)}{n(n-1)}.$$
Note that $p(2,\infty)=\lim_{n \to \infty} p(2,n)=1/4<1/e=p(1,\infty)$.

In this paper, we consider the problem of optimally selecting the $k$-th best candidate for general $k$.
Let $p(k,n):=\sup_{\tau \in \mathcal{M}_n} P(R_{\tau}=k)$, the maximum probability of selecting
the $k$-th best of $n$ candidates.   Szajowski \cite{ref26}  derived the asymptotic solutions as $n \to \infty$ for $k=3,4,5$.
Rose \cite{ref24} dealt with the case $k=(n+1)/2$ for odd $n$, which was called the  median problem and suggested by M. DeGroot
with the motivation of selecting a candidate representative of the entire sequence. (The  candidate of rank $k=(n+1)/2$ is, in some sense, representative of
all candidates.)  In the next section, we solve
the case $k=3$ for all finite $n\ge 3$ by showing ({\it cf.} Theorem 2.1) that the stopping rule
$\tau_{3,n}=\min\{a_n \le j \le n: X_j=2\} \wedge \min\{b_n \le j \le n: X_j=3\}$ attains the maximum probability $P(R_{\tau_{3,n}}=3)=p(3,n)$
for $n\ge 3$, where $x \wedge y:=\min\{x,y\}$ and the two thresholds $a_n<b_n$ are given in (\ref{v4e4}) and (\ref{v4e1}), respectively. In Section 3, we prove ({\it cf.} Theorems 3.1 and 3.2)
that (i) $p(1,n)=p(n,n)> p(k,n)$ for $1<k<n$,
(ii)  $p(k,n)\ge p(k,n+1)$, (iii) $p(k,n)\ge p(k+1,n+1)$, and  (iv) $p(k,\infty):=\lim_{n\to \infty} p(k,n)$ is decreasing in $k$. It is also noted
({\it cf.} Remark 3.1) that  the inequality
$p(k,n)\ge p(k+1,n)$ occasionally fails to hold for $k$ close to (but less than) $\lceil\frac{n}{2}\rceil$.
Furthermore, we extend  the result $p(1,n)=p(n,n)>p(k,n)$ for $1<k<n$
to the setting where the goal is to select a candidate whose absolute rank belongs to a prescribed subset $\Gamma$ of $\{1,\dots,n\}$
with $|\Gamma|=c \;(1\le c<n)$ ({\it cf.} Suchwalko and Szajowski \cite{ref25}). It is shown ({\it cf.} Theorem 3.3) that the probability of optimally selecting a candidate whose absolute rank belongs to $\Gamma$
is maximized when $\Gamma=\{1,\dots,c\}$ or $\Gamma=\{n-c+1,\dots,n\}$.
The proofs of several technical lemmas are
relegated to Section 4. Section 5 contains a computer program in Mathematica for verification of Theorem \ref{st4.1} for $3\le n\le 31$. It should be remarked that the optimal stopping rule is not necessarily unique. For example, a slight modification $\tau_{2,n}'$ of the optimal stopping rule $\tau_{2,n}$ also attains the maximum probability $p(2,n)$ where $\tau_{2,n}'\ge r_n'-1$ is given by $\tau_{2,n}'=r_n'-1$ if $X_{r_n'-1}=1$ and $\tau_{2,n}'=\tau_{2,n}$ otherwise.
The uniqueness issue of the optimal stopping rule is not addressed in this paper.

\section{Maximizing the probability of selecting the $k$-th best candidate with $k=3$}
\hspace*{18pt} We adopt the setup and notations in Ferguson \cite[Chapter 2]{ref5}. As defined in Section 1, $X_j$ is
the relative rank of the $j$-th candidate among the first $j$ candidates and $R_j$ is the absolute rank. Given $X_1=x_1,X_2=x_2,\dots,X_j=x_j$, $1\le j\le n$, let $y_j(x_1,x_2,\dots,x_j)$ be the return for stopping at stage $j$ (i.e. accepting the $j$-th candidate) and $V_j(x_1,x_2,\dots,x_j)$ the maximum return by optimally stopping from stage $j$ onwards. In other words, $y_j(x_1,x_2,\dots,x_j)$ is the conditional probability of $R_j=k$ (given $X_i=x_i$, $1\le i\le j$), which defines the reward function for the stopping problem of optimally selecting the $k$-th best candidate. Given $X_i=x_i$, $1\le i\le j$, $V_j(x_1,x_2,\dots,x_j)$ is the (maximum) expected reward by optimally stopping from stage $j$ onwards. Then
$V_n(x_1,x_2,\dots,x_n)=y_n(x_1,x_2,\dots,x_n)$, and
\begin{equation}\label{se2.1}
V_j\left(x_1,\dots,x_j\right)=\max\left\{y_j(x_1,\dots,x_j),E\left(V_{j+1}\left(x_1,\dots,x_j,X_{j+1}\right)\Big|X_1=x_1,\dots,X_j=x_j\right)\right\},
\end{equation}
for $j=n-1,n-2,\dots,1$. Given $X_i=x_i, i=1,\dots,j$, it is optimal to stop at stage $j$ if $V_j\left(x_1,x_2,\dots,x_j\right)=y_j(x_1,x_2,\dots,x_j)$ and to continue otherwise. The (optimal) value of the stopping problem is $V_1(1)$, i.e. $V_1(1)=\sup_{\tau\in\mathcal{M}_n}P(R_{\tau}=k)$. This formalizes the method of backward induction. See also Chow, Robbins and Siegmund \cite{ref3}.

It is well known that $X_1,X_2,\dots, X_n$ are independent and $X_j$ has a uniform distribution over $\{1,2,\dots,j\}$.
Given $X_i=x_i, i=1,\dots,j$, the conditional probability  of $R_j=k$ is the same as the probability that a random sample of size $j$ contains the $k$-th best candidate whose (relative) rank in the sample is $x_j$; thus
\begin{equation}\label{v11}
P(R_j=k|X_1=x_1,X_2=x_2,\dots,X_j=x_j)=\frac{{k-1\choose x_j-1}{n-k\choose j-x_j}}{{n\choose j}},
\end{equation}
where we adopt the usual convention that ${m\choose \ell}=0$ for $m<\ell$.

From the independence of $X_1,X_2,\dots,X_n$, the conditional expectation on the right hand side of (\ref{se2.1}) reduces to $E(V_{j+1}(x_1,x_2,\dots,x_j,X_{j+1}))$. Note also that $y_j(x_1,\dots,x_j)$ depends only on  $x_j$ ({\it cf.} (\ref{v11})), and so does $V_j(x_1,\dots,x_j)$. Hence,
we have
\begin{align}
&V_n(x_n)=y_n(x_n)\notag\\
\mbox{and}\;\;\;&V_j(x_j)=\max\left\{y_j(x_j),\frac{1}{j+1}\sum_{i=1}^{j+1}V_{j+1}(i)\right\}\;\;\mbox{for}\;\;j=n-1,n-2,\dots,1.\label{b4}
\end{align}
Thus, it is optimal to stop at the first $j$ with
\begin{equation*}
y_j(x_j)\ge\frac{1}{j+1}\sum_{i=1}^{j+1}V_{j+1}(i).
\end{equation*}

For the problem of optimally selecting the $k$-th best candidate with $k=3$,
we have $y_j(x_j)=P(R_j=3|X_1=x_1,\dots,X_j=x_j)$, which equals  ({\it cf.} (\ref{v11}))
\begin{equation}\label{e15}
y_j(x_j)=
\begin{cases}
\displaystyle\frac{j(n-j-1)(n-j)}{n(n-1)(n-2)},&\mbox{if}\;\;x_j=1;\\[4mm]
\displaystyle\frac{2j(j-1)(n-j)}{n(n-1)(n-2)},&\mbox{if}\;\;x_j=2;\\[4mm]
\displaystyle\frac{j(j-1)(j-2)}{n(n-1)(n-2)},&\mbox{if}\;\;x_j=3;\\[4mm]
0,&\mbox{otherwise}.
\end{cases}
\end{equation}
Setting $\sum_{i=\ell}^m c_i:=0$ whenever $\ell>m$, define for $n\ge3$,
\begin{align}
b_n&=\min\left\{j=2,3,\dots,n:\;\sum_{i=j+1}^n\frac{1}{i-2}\le\frac{1}{2}\right\},\label{v4e1}\\
u_n&=(b_n-2)(2n-4)\sum_{i=b_n}^n\frac{1}{i-2},\label{v4e2}\\
f_n(x)&=3x^2-(1+4n)x+(n-2)b_n+2(n+1)+u_n,\label{v4e3}\\
a_n&=\min\left\{j=2,3,\dots,n:\;f_n(j)\le0\right\}.\label{v4e4}
\end{align}

\begin{rem}
Note that  $3 \le b_n\le b_{n+1} \le b_n+1$ for
$n\ge 3$, implying that $f_n(1)>0$ for all $n\ge3$.
In order for $a_n$ in $(\ref{v4e4})$ to be well defined, we need to show that the second-order polynomial equation $f_n(x)=0$ has two real roots $x_0<y_0$ with $\lceil x_0\rceil\le y_0$ $(\mbox{so that}\; a_n=\lceil x_0\rceil)$. For $3\le n\le 31$, this can be verified by numerical computations. For $n\ge32$, we have $b_n<\frac{2n-1}{3}$ and $u_n\le(n-2)b_n$ $({\it cf.} \;(\ref{ae8})\;\mbox{and} \;(\ref{ae6}))$, implying that $f_n(\frac{2n-1}{3})<0$ and $f_n(\frac{2n+2}{3})<0$. So, $x_0<\frac{2n-1}{3}$, implying that $\lceil x_0\rceil<\frac{2n+2}{3}<y_0$.
With a little effort, it can be shown that $2 \le a_n\le a_{n+1}\le a_n+1$ for $n\ge 3$.
\end{rem}

The next theorem is our main result.

\begin{thm}\label{st4.1}
For $n\ge 3$, we have $a_n<b_n$. Furthermore, the stopping rule
\begin{equation*}
\tau_{3,n}=\min\{a_n\le j \le n: X_j=2\} \wedge \min\{b_n \le j \le n: X_j=3\}
\end{equation*}
maximizes the probability of selecting the 3rd best candidate.
\end{thm}

Figure 1 illustrates the optimality of $\tau_{3,n}$ for the case $n=13$ with $a_{13}=7$ and $b_{13}=9$. With the help of a computer program in Mathematica, we have verified Theorem \ref{st4.1} for $3\le n\le 31$ by numerically evaluating $V_j(x_j)$, $j=n,n-1,\dots,1$. (For completeness, the computer program is provided in Section 5.) While it seems intuitively reasonable for the optimal stopping rule $\tau_{3,n}$ to involve two thresholds for general $n$, the exact expressions for the thresholds $a_n$ and $b_n$ in (\ref{v4e4}) and (\ref{v4e1}) were found by some guesswork and tedious analysis. To prove Theorem \ref{st4.1} for $n \ge 32$, we need the following lemmas whose proofs are relegated to Section 4.
\begin{figure}
\includegraphics[width=\textwidth,height=10cm]{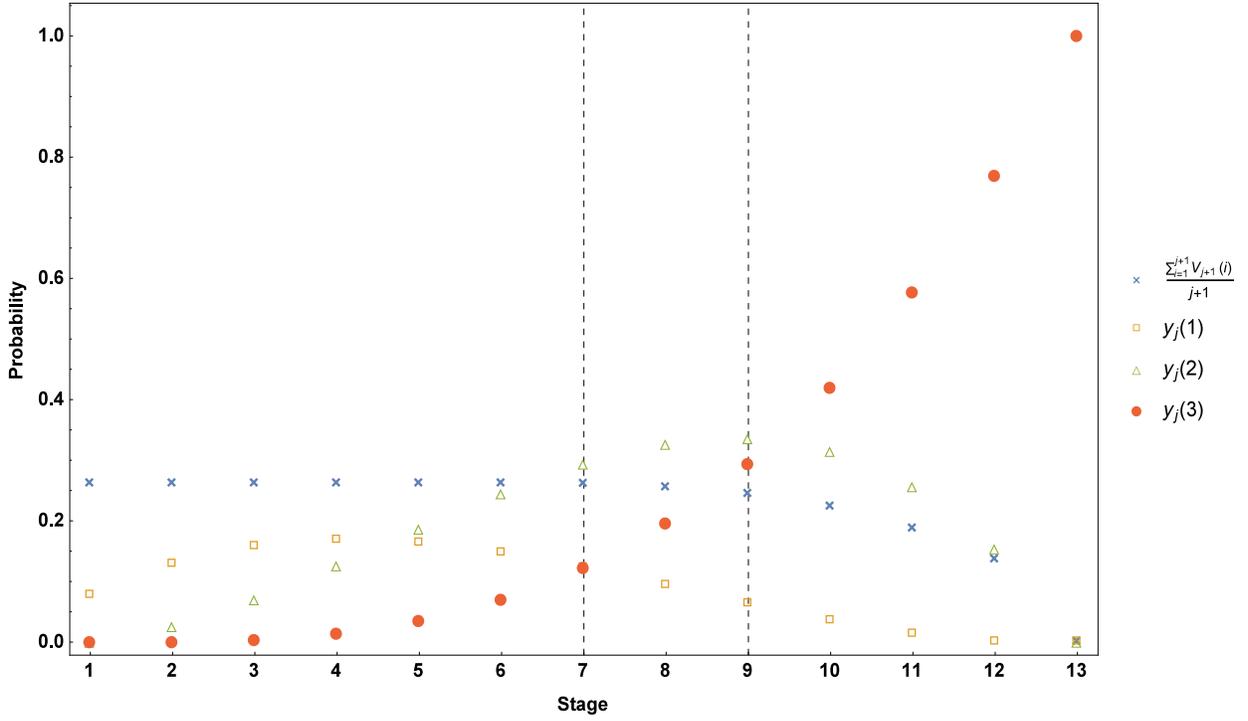}
\caption{The optimality of $\tau_{3,13}$.}
\end{figure}
\begin{lem}\label{sl4.2}
Let $y_0$ be the larger root of the second-order polynomial equation $f_n(x)=0$.
 Then for $n\ge32$, we have {\upshape(i)} $a_n<b_n$; {\upshape(ii)} $b_n<y_0$; {\upshape(iii)} $a_n>(n+4)/3$.
\end{lem}
\begin{lem}\label{sl4.4}
Given $X_1=x_1,X_2=x_2,\dots,X_j=x_j$, let $h_j(x_j)=h_j(x_1,x_2,\dots,x_j)$ be the conditional probability of selecting the 3rd best candidate when $\tau_{3,n}$ is used for stages $j,j+1,\dots,n$. Then for $n\ge32$,
\begin{itemize}
\item[\upshape(i)]
$$\hspace*{-4mm}
h_j(x_j)=
\begin{cases}
\displaystyle\frac{(a_n-1)\left[a_n^2-(1+2n)a_n+(n-2)b_n+2(n+1)+u_n\right]}{n(n-1)(n-2)},&\mathrm{if}\;\;j<a_n;\\
y_j(2),&\mathrm{if}\;\;j\ge a_n\;\mathrm{and}\;x_j=2;\\[1mm]
\displaystyle\frac{j\left[j^2+(1-2n)j+(n-2)b_n+2+u_n\right]}{n(n-1)(n-2)},&\mathrm{if}\;\;a_n\le j\le b_n-1\;\mathrm{and}\;x_j\neq 2;\\
y_j(3),&\mathrm{if}\;\;j\ge b_n\;\mathrm{and}\;x_j=3;\\[1mm]
\displaystyle\frac{j(j-1)}{n(n-1)(n-2)}\left[(2n-4)\sum_{i=j+1}^{n}\frac{1}{i-2}-(n-j)\right],&\mathrm{if}\;\;j\ge b_n\;\mathrm{and}\;x_j\neq 2,3.
\end{cases}
$$
\item[\upshape(ii)]
$$\hspace*{-4mm}
\frac{1}{j+1}\sum_{i=1}^{j+1}h_{j+1}(i)=
\begin{cases}
\displaystyle\frac{(a_n-1)\left[a_n^2-(1+2n)a_n+(n-2)b_n+2(n+1)+u_n\right]}{n(n-1)(n-2)},&\mathrm{if}\;\;j<a_n;\\[4mm]
\displaystyle\frac{j\left[j^2+(1-2n)j+(n-2)b_n+2+u_n\right]}{n(n-1)(n-2)},&\mathrm{if}\;\;a_n\le j\le b_n-1;\\[4mm]
\displaystyle\frac{j(j-1)}{n(n-1)(n-2)}\left[(2n-4)\sum_{i=j+1}^{n}\frac{1}{i-2}-(n-j)\right],&\mathrm{if}\;\;b_n\le j\le n-1.
\end{cases}
$$
\end{itemize}
\end{lem}

\begin{lem}\label{sl4.5}
For $n\ge32$, $1\le j<a_n$ and $1\le x_j\le j$, we have
\begin{equation*}
y_j(x_j)<\frac{1}{j+1}\sum_{i=1}^{j+1}h_{j+1}(i).
\end{equation*}
\end{lem}

\begin{lem}\label{sl4.6}
For $n\ge32$ and $a_n\le j<b_n$, we have
\upshape(i) $y_j(2)\ge\frac{1}{j+1}\sum_{i=1}^{j+1}h_{j+1}(i)$;
\upshape(ii) $y_j(1)<\frac{1}{j+1}\sum_{i=1}^{j+1}h_{j+1}(i)$;
\upshape(iii) $y_j(3)<\frac{1}{j+1}\sum_{i=1}^{j+1}h_{j+1}(i)$.
\end{lem}

\begin{lem}\label{sl4.7}
For $n\ge32$ and $b_n\le j\le n-1$, we have
\upshape(i) $y_j(1)<\frac{1}{j+1}\sum_{i=1}^{j+1}h_{j+1}(i)$; (ii) $y_j(2)\ge\frac{1}{j+1}\sum_{i=1}^{j+1}h_{j+1}(i)$; (iii) $y_j(3)\ge\frac{1}{j+1}\sum_{i=1}^{j+1}h_{j+1}(i)$.
\end{lem}

\begin{proof}[\bf Proof of Theorem \ref{st4.1}]
As remarked before, the theorem has been verified for $3\le n\le31$ by numerical computations.
For $n\ge32$, we need to show that $h_j$ satisfies
\begin{equation}\label{e14}
h_j(x_j)=\max\left\{y_j(x_j),\frac{1}{j+1}\sum_{i=1}^{j+1}h_{j+1}(i)\right\}\;\mbox{for}\;\;1\le j<n.
\end{equation}
Since $h_j(x_j)$ is the conditional probability of selecting the 3rd best candidate when $\tau_{3,n}$ is used for stages $j,\dots,n$, we have $h_j(x_j)=\frac{1}{j+1}\sum_{i=1}^{j+1}h_{j+1}(i)$ if either ($j<a_n$) or ($a_n\le j<b_n$ and $x_j\neq 2$) or ($b_n\le j<n$ and $x_j\neq 2,3$), which together with Lemmas \ref{sl4.5} -- \ref{sl4.7} establishes (\ref{e14}).
\end{proof}

\begin{rem}\label{r1}
Let $d_1=\lim_{n\to\infty} a_n/n$ and $d_2=\lim_{n\to\infty} b_n/n$. It is shown in Section 4 that
\begin{equation}\label{b5}
d_1=\frac{2}{2\sqrt{e}+\sqrt{4e-6\sqrt{e}}}\approx0.466\;\;\mbox{and}\;\;d_2=\frac{1}{\sqrt{e}}\approx0.606.
\end{equation}
It is also shown in Section 4 that as $n \to \infty$, $h_1(1)=p(3,n)$, the maximum probability of selecting the 3rd best candidate, tends to
\begin{equation}\label{b6}
p(3,\infty)=2d_1^2(1-d_1)=\frac{8\left(2\sqrt{e}-2+\sqrt{4e-6\sqrt{e}}\right)}{\left(2\sqrt{e}+\sqrt{4e-6\sqrt{e}}\right)^3}.
\end{equation}
Note that $p(3,\infty)\approx0.232<0.25=p(2,\infty)$. These limiting results agree with the asymptotic solution for $k=3$ in Szajowski \cite{ref26}.
\end{rem}

\section{Some  results on $p(k,n)$ and $p(k,\infty)$}
\hspace*{18pt}In this section, we present several inequalities for $p(k,n)$ and $p(k,\infty):=\lim_{n \to \infty} p(k,n)$.

\begin{thm}\label{t1}
For $n\ge 3$ and $1<k<n$, we have
$p(1,n)=p(n,n)>p(k,n)$.
\end{thm}
\begin{proof}
By symmetry, $p(1,n)=p(n,n)$. (More generally, $p(k,n)=p(n-k+1,n)$.) For the problem of  selecting the $k$-th best candidate ($1<k<n$),
 a (non-randomized) optimal stopping rule $\tau$
 is determined by a sequence of subsets $\{S_j\}$ such that
 $S_j\subset \{1,2,\dots,j\}\; (j=1,\dots,n)$ and
$
\tau=\min\{j:X_j\in S_j\}$. Since stopping at $n$ is enforced (if $\tau>n-1$), we may assume that $S_n=\{1,2,\dots,n\}$. Thus,
\begin{equation}
P(R_{\tau}=k)=p(k,n). \label{e0}
\end{equation}
 Define, for $j=1,\dots, n-1$,
\begin{equation*}
S_{j}'=
\begin{cases}
\emptyset,&\;\;\mbox{if}\;S_j=\emptyset;\\
\{1\},&\;\;\mbox{if}\;S_j\ne \emptyset;
\end{cases}
\end{equation*}
 and  $S_n'=\{1,2,\dots,n\}$. Let $\tau'=\min\{j: X_j \in S_j'\}$, which, as a stopping rule, may be applied to selecting the best candidate.
 Thus
 \begin{equation}
 P(R_{\tau'}=1)\le \sup_{\nu \in \mathcal{M}_n} P(R_{\nu}=1)= p(1,n). \label{e00}
 \end{equation}
Note that for $j=1,\dots,n$,
\begin{align}
P(R_j=1,X_j=1)=\frac{1}{n}
&=P(R_j=k)\notag\\
&\geq P(R_j=k,X_j\in S_{j}).\label{e1}
\end{align}
By (\ref{v11}), given $X_1=x_1,\dots, X_j=x_j$, the conditional distribution of $R_j$ depends  only on $x_j$, implying that $X_1,\dots, X_{j-1}$ and $(X_j,R_j)$ are independent. So
 if $S_j\ne \emptyset$,
\begin{align}
P(\tau=j,R_j=k)&=P(X_i\notin S_i,\;i=1,\dots,j-1,\;X_j\in S_j,R_j=k)\notag\\
&=\left[\prod_{i=1}^{j-1}P(X_i\notin S_i)\right]  P(X_j\in S_j, R_j=k)\notag\\
&\le\left[\prod_{i=1}^{j-1}P(X_i\notin S_i')\right] P(X_j=1,R_j=1)\label{e2}\\
&=P(\tau'=j,R_j=1)\notag,
\end{align}
where the inequality follows from  (\ref{e1}) and  $|S_i'|\le|S_i|$ for all $i$.  (If $S_j=\emptyset$, then
$P(\tau=j,R_j=k)=P(\tau'=j,R_j=1)=0$.) By (\ref{e0}), (\ref{e00}) and (\ref{e2}),  we have
\begin{align}
p(k,n)=P(R_\tau=k)
&=\sum_{j=1}^nP(\tau=j,R_j=k)\notag\\
&\le\sum_{j=1}^nP(\tau'=j,R_j=1)=P(R_{\tau'}=1)\le p(1,n).\label{e000}
\end{align}

It remains to show that (at least) one of the two inequalities in (\ref{e000}) is strict (so that $p(k,n)<p(1,n)$).
If the stopping rule $\tau'$ is not optimal for selecting the best candidate, then the second inequality in (\ref{e000}) is strict.
 Suppose $\tau'$ is optimal for selecting the best candidate, which implies, in view of $n \ge 3$, that $S_1'=\emptyset$
and $S_{n-1}'=\{1\}$, which in turn implies that $|S_{n-1}|\ge 1$. If $|S_{n-1}|\ge 2$, then the inequality in (\ref{e2})
is strict for $j=n$, implying that the first inequality in (\ref{e000}) is strict.
Suppose $S_{n-1}=\{\ell\}$ for some $\ell$. Then we have
\begin{equation*}
P(R_{n-1}=k, X_{n-1}=\ell)=
\begin{cases}
\frac{n-k}{n(n-1)},&\;\;\mbox{if}\;k=\ell;\\
\frac{k-1}{n(n-1)},&\;\;\mbox{if}\;k=\ell+1;\\
0,&\;\;\mbox{if}\;k-\ell\ne 0,1;
\end{cases}
\end{equation*}
implying, in view of $1<k<n$, that the inequality in (\ref{e1}) is strict for $j=n-1$, which in turn implies that the inequality in (3.4)
is strict for $j=n-1$. It follows  that the first inequality in (\ref{e000}) is strict. The proof is complete.
\end{proof}

\begin{thm}
For $1\le k \le n$, we have $p(k,n)\ge p(k,n+1)\; (\mbox{i.e.}\; p(k,n) \;\mbox{is decreasing in}\; n)$ and $p(k,n)\ge p(k+1,n+1)$. Furthermore,
$p(k,\infty):=\lim_{n \to \infty} p(k,n)$ is well defined, and $p(k,\infty)\ge p(k+1,\infty)$.
\end{thm}
\begin{proof}
(i) To show $p(k,n)\ge p(k,n+1)$, consider the case of selecting the $k$-th best of $n+1$ candidates. Let
the random variable $I \in \{1,\dots, n+1\}$ be such that $R_I=n+1$ (i.e. the worst candidate is the $I$-th person to be interviewed). If $I$ is known to the manager (or more precisely, the manager knows the position of the worst candidate before the interview process begins), then
the problem of optimally selecting the $k$-th best of the $n+1$ candidates is equivalent  to that of optimally
selecting the $k$-th best of the $n$ candidates (excluding the worst one).
(Indeed, let $X_i'=X_i$ for $1\le i<I$ and $X_i'=X_{i+1}$ for $I\le i\le n$. Given $I$, $X_1',\dots, X_n'$ are (conditionally) independent
with each $X_i'$ being uniform over $\{1,\dots,i\}$.)
Thus, when $I$ is known to the manager,
the maximum probability
of selecting the $k$-th best candidate equals $p(k,n)$, which must be at least as large as $p(k,n+1)$, the maximum probability
of selecting the $k$-th best of the $n+1$ candidates when  $I$ is unavailable. This proves that $p(k,n)\ge p(k,n+1)$.

(ii) To show $p(k,n)\ge p(k+1,n+1)$, note that
\begin{equation}
p(k,n)=p(n-k+1,n)\ge p(n+1-k,n+1)=p(k+1,n+1),\label{e12}
\end{equation}
where the two equalities follow from the symmetry property $p(k,n)=p(n-k+1,n)$ and the inequality follows from
the decreasing property of $p(k,n)$ in $n$.

(iii) Since $p(k,n)$ is decreasing in $n$, $p(k,\infty):=\lim_{n \to \infty} p(k,n)$ is well defined.
By (\ref{e12}), we have
$$
p(k,\infty)=\lim_{n\to\infty}p(k,n)\ge\lim_{n\to\infty}p(k+1,n+1)=p(k+1,\infty).
$$
The proof is complete.
\end{proof}
\begin{rem}
We conjecture that the three inequalities in Theorem 3.2 are all strict.  While $p(k,n)$ is decreasing in $n$,
in view of $p(1,n)>p(k,n)$ for $1<k<n$ and $p(k,\infty)\ge p(k+1,\infty)$, it may be tempting to conjecture that $p(k,n)\ge p(k+1,n)$
for $1\le k <\lceil\frac{n}{2}\rceil$. However, this inequality occasionally fails to hold for $k$ close to $(\mbox{but less than})$ $\lceil\frac{n}{2}\rceil$.
Our numerical results show that the set  $\{(k,n): 1\le k <\lceil\frac{n}{2}\rceil, n\le 50, p(k,n)<p(k+1,n)\}$ consists of $(2,5),(2,7),(7,15),(9,19),(10,21),(12,25),(21,43),(22,47),(24,49)$ and $(24,50)$.
Moreover, it can be shown that $p(2,n)>p(3,n)$ for all $n\ge 8$.  Let $\rho=\liminf_{n \to \infty} K(n)/n$
  where $K(n)=\max\{1\le k \le \lceil\frac{n}{2}\rceil: p(1,n)\ge p(2,n)\ge \cdots \ge p(k,n)\}$. While $0\le \rho \le 1/2$, it appears to be a
  challenging task to find the exact value of $\rho$. Our limited
  numerical results suggest that $\rho$ may be equal to $1/2$.
\end{rem}
\begin{rem}
It may be of interest to see how fast $p(k,\infty)$ tends to $0$ as $k$ increases. By considering some suboptimal rules, we
have derived a crude lower bound $k^{\frac{-k}{k-1}}$ for $p(k,\infty)$. The details are omitted.
\end{rem}
The next theorem extends Theorem \ref{t1} to the setting where the goal is to select a candidate whose rank belongs to a prescribed subset
$\Gamma$ of $\{1,\dots,n\}$ ({\it cf.} Suchwalko and Szajowski \cite{ref25}). Let $$p(\Gamma,n)=\sup_{\tau \in \mathcal{M}_n} P(R_\tau \in \Gamma).$$

\begin{thm}\label{t2}
For any subset $\Gamma$ of $\{1,2,\dots,n\}$ with $|\Gamma|=c\;\;(1\le c <n)$, we have
$$
p(\Gamma,n)\le p(\{1,2,\dots,c\},n)=p(\{n-c+1,\dots,n\},n).
$$

\end{thm}

In the proof below, it is convenient to take the convention that ${0\choose 0}:=1$ and ${n\choose k}:=0$ if $n<k$ or $n<0$ or $k<0$, so that
\begin{equation}\label{e3}
{n\choose k}={n-1\choose k}+{n-1\choose k-1}\;\;\mbox{for}\;\;(k,n)\in\mathbb{Z}\times\mathbb{Z}\backslash\{(0,0)\},
\end{equation}
and
\begin{equation}\label{e4}
{n\choose k}\ge{n-1\choose k}+{n-1\choose k-1}\;\;\mbox{for}\;\;(k,n)\in\mathbb{Z}\times\mathbb{Z},
\end{equation}
where $\mathbb{Z}$ is the set of all integers.

\begin{proof}[\bf Proof of Theorem \ref{t2}]
As in  the proof of Theorem \ref{t1}, let $\tau$ be a (non-randomized) optimal stopping rule determined by a sequence of
subsets $\{S_j\}$ of $\{1,\dots,n\}$ such that $S_j \subset \{1,\dots,j\}$, $\tau=\min\{j:X_j \in S_j\}$ and $P(R_{\tau} \in \Gamma)=p(\Gamma,n)$.
Again, as stopping at $n$ is enforced (if $\tau>n-1$), we may assume that $S_n=\{1,2,\dots,n\}$.
Let $S_j'=\{1,2,\dots, |S_j|\}$, so $|S_j'|=|S_j|$ (in particular, $S_j'=\emptyset$ if $S_j=\emptyset$). Let $\tau'=\min\{j: X_j \in S_j'\}$.
Claim
\begin{equation}\label{e6}
P(R_j\in\{t_1,t_2,\dots,t_c\},X_j\in\{s_1,s_2,\dots,s_d\})\le P(R_j\in\{1,2,\dots,c\},X_j\in\{1,2,\dots,d\})
\end{equation}
for $1\le d \le j\le n$, $1\le c\le n, 1\le t_1<t_2<\dots<t_c\le n$, and $1\le s_1<s_2<\dots<s_d\le j$.
If the claim (\ref{e6}) is true, then for $j=1,\dots,n$,
\begin{align}
P(\tau=j,R_j\in \Gamma)&=P(X_i\notin S_i,\;i=1,\dots,j-1,\;X_j\in S_j,R_j\in \Gamma)\notag\\
&=\left[\prod_{i=1}^{j-1}P(X_i\notin S_i)\right]  P(R_j \in \Gamma, X_j\in S_j)\notag\\
&\le\left[\prod_{i=1}^{j-1}P(X_i\notin S_i')\right] P(R_j \in \{1,\dots,c\}, X_j\in S_j')\;\;(\mbox{by}\;(\ref{e6}))\notag\\
&=P(X_i\notin S_i', i=1,\dots, j-1, X_j \in S_j', R_j \in \{1,\dots,c\})\notag\\
&=P(\tau'=j,R_j\in \{1,\dots,c\})\notag,
\end{align}
implying that $p(\Gamma,n)=P(R_\tau \in \Gamma)\le  P(R_{\tau'}\in \{1,\dots,c\})\le p(\{1,\dots,c\},n)$.

It remains to establish  (\ref{e6}). Note that
\begin{align*}
P(R_j \in \{t_1,\dots, t_c\}, & \; X_j \in \{s_1,\dots,s_d\})\\
& \le P(R_j \in \{t_1,\dots, t_c\})=\frac{c}{n}\\
&=P(R_j \in \{1,\dots,c\})\\
&=P(R_j \in \{1,\dots,c\}, X_j \in \{1,\dots,d\})\;\;\mbox{(if}\; d\ge c),
\end{align*}
showing that (\ref{e6}) holds for $d\ge c$.
Since
$$
P(R_j=a,X_j=b)=\frac{{a-1\choose b-1}{n-a\choose j-b}}{n{n-1\choose j-1}}\;\;\mbox{for all integers}\;\;a>0, b>0,
$$
(\ref{e6}) is equivalent to
\begin{equation}\label{e7}
\sum_{i=1}^d\sum_{\ell=1}^c{t_\ell-1\choose s_i-1}{n-t_\ell\choose j-s_i}\le\sum_{i=1}^d\sum_{\ell=1}^c{\ell-1\choose i-1}{n-\ell\choose j-i},
\end{equation}
for $1\le d\le j \le n, 1\le c\le n, 1\le t_1<\cdots<t_c\le n$ and $1\le s_1<\cdots<s_d\le j$.
Note that (\ref{e7}) holds for $d \ge c$ (since (\ref{e6})  does for $d\ge c$).
Also, from ${n-t_\ell\choose j-s_i}=0$ for $t_\ell>n$ or $s_i> j$, it follows easily that for fixed $n$, if (\ref{e7}) holds for all
$(j,c,d, t_1,\dots, t_c, s_1,
\dots,s_d)$ with
 $1\le d\le j \le n, 1\le c\le n, 1\le t_1<\cdots<t_c\le n$ and $1\le s_1<\cdots<s_d\le j$, then (\ref{e7}) holds for all
 $(j,c,d,t_1,\dots,t_c,s_1,\dots,s_d)$ with $1\le j \le n, 1\le t_1<\cdots<t_c$ and $1\le s_1<\cdots<s_d$. This (trivial) observation is needed later.
  To prove (\ref{e7}),  we proceed by induction on $n$. For $n=1$, necessarily $j=1$ and $c=d=1$ (since $1\le d\le j\le n$
and $1\le c\le n$). So (\ref{e7}) holds for $n=1$.

Suppose (\ref{e7}) holds for (fixed) $n\ge1$ and for all $(j,c,d,t_1,\dots,t_c,s_1,\dots,s_d)$ with $1\le d\le j \le n, 1\le c\le n, 1\le t_1<\cdots<t_c\le n$ and $1\le s_1<\cdots<s_d\le j$ (and hence
for all $(j,c,d,t_1,\dots,t_c,s_1,\dots,s_d)$ with
$1\le j \le n, 1\le t_1<\cdots<t_c$ and $1\le s_1<\cdots<s_d$).
We need to show that (\ref{e7}) holds for $n+1$ (with $1\le d<c$), i.e.
\begin{equation}\label{e8}
\sum_{i=1}^d\sum_{\ell=1}^c{t_\ell-1\choose s_i-1}{n-t_\ell+1\choose j-s_i}\le\sum_{i=1}^d\sum_{\ell=1}^c{\ell-1\choose i-1}{n-\ell+1\choose j-i},
\end{equation}
for $1\le d \le j\le n+1, 1\le d < c \le n+1$, $1\le t_1<t_2<\dots<t_c\le n+1$ and $1\le s_1<s_2<\dots<s_d\le j$. If $j=1$, then
necessarily $d=1$ and $s_1=1$, so that both sides of (\ref{e8}) equal $c$, implying that   (\ref{e8}) holds for $j=1$. For $j=n+1$, the left hand side of (\ref{e8}) equals
$$
\sum_{i=1}^d\sum_{\ell=1}^c{t_\ell-1\choose s_i-1}{n-t_\ell+1\choose n-s_i+1}\le d,
$$
since the two inequalities $t_\ell-1\ge s_i-1$  and $n-t_\ell+1\ge n-s_i+1$  hold simultaneously if and only if $t_\ell=s_i$.
 The right hand side of (\ref{e8}) equals
$$
\sum_{i=1}^d\sum_{\ell=1}^c{\ell-1\choose i-1}{n-\ell+1\choose n-i+1}=d,
$$
since ${\ell-1\choose i-1}{n-\ell+1\choose n-i+1}=1$ or $0$ according to whether $i=\ell$ or $i \ne \ell$.
Thus, (\ref{e8}) holds for $j=n+1$.

We now consider $2\le j\le n$. Suppose $n-t_{c}+1=j-s_{d}=0$. Then the left hand side of (\ref{e8}) equals
\begin{align}
&\sum_{i=1}^{d}\sum_{\ell=1}^{c-1}{t_\ell-1\choose s_i-1}{n-t_\ell+1\choose j-s_i}+{n\choose j-1}\notag\\
=&\sum_{i=1}^{d}\sum_{\ell=1}^{c-1}{t_\ell-1\choose s_i-1}\left[{n-t_\ell\choose j-s_i}+{n-t_\ell\choose j-s_i-1}\right]+{n\choose j-1}\;\;(\mbox{by}\;(\ref{e3}))\notag\\
=&\sum_{i=1}^{d}\sum_{\ell=1}^{c-1}{t_\ell-1\choose s_i-1}{n-t_\ell\choose j-s_i}+\sum_{i=1}^{d-1}\sum_{\ell=1}^{c-1}{t_\ell-1\choose s_i-1}{n-t_\ell\choose (j-1)-s_i}+{n\choose j-1}.\label{e9}
\end{align}
By the induction hypothesis (applied to each of the two double sums), (\ref{e9}) is less than or equal to
\begin{align*}
&\sum_{i=1}^d\sum_{\ell=1}^{c-1}{\ell-1\choose i-1}{n-\ell\choose j-i}+\sum_{i=1}^{d-1}\sum_{\ell=1}^{c-1}{\ell-1\choose i-1}
{n-\ell\choose (j-1)-i}+{n\choose j-1}\\
=&\sum_{i=1}^{d-1}\sum_{\ell=1}^{c-1}{\ell-1\choose i-1}\left[{n-\ell\choose j-i}+{n-\ell\choose j-i-1}\right]+\sum_{\ell=1}^{c-1}{\ell-1\choose d-1}{n-\ell\choose j-d}+{n\choose j-1},
\end{align*}
which by (\ref{e3}) is  equal to
\begin{align}
&\sum_{i=1}^{d-1}\sum_{\ell=1}^{c-1}{\ell-1\choose i-1}{n-\ell+1\choose j-i}+\sum_{\ell=d}^{c-1}{\ell-1\choose d-1}{n-\ell\choose j-d}+{n\choose j-1}\label{e10}.
\end{align}
We need the following identity
\begin{equation}\label{e313}
\sum_{i=d+1}^{c}{c-1\choose i-1}{n-c+1\choose j-i}=\sum_{\ell=d}^{c-1}{\ell-1\choose d-1}{n-\ell\choose j-d-1},
\end{equation}
which holds by observing that the left hand side is the total number of subsets of $\{1,\dots,n\}$ with $j-1$ elements and with the $d$-th smallest element less than $c$ while the term ${\ell-1\choose d-1}{n-\ell\choose j-d-1}$ on the right hand side is the number of subsets of $\{1,\dots,n\}$ with $j-1$ elements
and with the $d$-th smallest element being $\ell$.
 In view of (\ref{e313}),
\begin{align}
{n\choose j-1}&=\sum_{i=1}^{d}{c-1\choose i-1}{n-c+1\choose j-i}+\sum_{i=d+1}^{c}{c-1\choose i-1}{n-c+1\choose j-i}\notag\\
&=\sum_{i=1}^{d}{c-1\choose i-1}{n-c+1\choose j-i}+\sum_{\ell=d}^{c-1}{\ell-1\choose d-1}{n-\ell\choose j-d-1}.\label{e314}
\end{align}
We have shown that  the left hand side of (\ref{e8}) is less than  or equal to (\ref{e10}), which by (\ref{e314}) equals
\begin{align*}
&\sum_{i=1}^{d-1}\sum_{\ell=1}^{c-1}{\ell-1\choose i-1}{n-\ell+1\choose j-i}+\sum_{\ell=d}^{c-1}{\ell-1\choose d-1}\left[{n-\ell\choose
j-d}+{n-\ell\choose j-d-1}\right]+\sum_{i=1}^{d}{c-1\choose i-1}{n-c+1\choose j-i}\\
=&\sum_{i=1}^{d-1}\sum_{\ell=1}^{c-1}{\ell-1\choose i-1}{n-\ell+1\choose j-i}+\sum_{\ell=d}^{c-1}{\ell-1\choose d-1}{n-\ell+1\choose
j-d}+\sum_{i=1}^{d}{c-1\choose i-1}{n-c+1\choose j-i}\;\;\mbox{(by\;(\ref{e3}))}\\
=&\sum_{i=1}^d\sum_{\ell=1}^c{\ell-1\choose i-1}{n-\ell+1\choose j-i},
\end{align*}
establishing (\ref{e8}) for the case that $2\le j \le n$ and $n-t_c+1=j-s_d=0$.

It remains to deal with the case that $2\le j\le n$ and $(n-t_c+1,j-s_d)\neq(0,0)$
(implying that $(n-t_\ell+1,j-s_i) \ne (0,0)$ for all $i, \ell$). By (\ref{e3}), the left hand side of (\ref{e8}) equals
\begin{align*}
&\sum_{i=1}^d\sum_{\ell=1}^c{t_\ell-1\choose s_i-1}{n-t_\ell\choose j-s_i}+\sum_{i=1}^d\sum_{\ell=1}^c
{t_\ell-1\choose s_i-1}{n-t_\ell\choose (j-1)-s_i}\\
\le&\sum_{i=1}^d\sum_{\ell=1}^c{\ell-1\choose i-1}{n-\ell\choose j-i}+\sum_{i=1}^d\sum_{\ell=1}^c{\ell-1\choose i-1}{n-\ell\choose (j-1)-i}\;\;\mbox{(by the induction hypothesis)}\\
=&\sum_{i=1}^d\sum_{\ell=1}^c{\ell-1\choose i-1}\left[{n-\ell\choose j-i}+{n-\ell\choose j-i-1}\right]\\
\le&\sum_{i=1}^d\sum_{\ell=1}^c{\ell-1\choose i-1}{n-\ell+1\choose j-i}\;\;\mbox{(by (\ref{e4}))}.
\end{align*}
Note that the first inequality follows from the induction hypothesis  applied to each of the two double sums
where $t_c>n$ or $s_d>j-1$ is possible. (Recall that the induction hypothesis applies to all
$(j,c,d,t_1,\dots,t_c,s_1,\dots,s_d)$ with $1\le j\le n, 1\le t_1<\cdots<t_c$ and $1\le s_1<\cdots<s_d$.) The proof is complete.
\end{proof}

\begin{rem}
As pointed out by a referee, the identities $(\ref{e313})$ and $(\ref{e314})$ are variants of Chu-Vandermonde convolution formula. $($See the first identity in Table 169 of Graham et al. \cite{ref0}.$)$
\end{rem}

\section{Proofs of Lemmas \ref{sl4.2}--\ref{sl4.7} and (\ref{b5})--(\ref{b6})}
\hspace*{18pt}To prove Lemmas 2.1--2.5, we need the following lemma.
\begin{lem}\label{sl4.1}
For $n\ge32$, we have
\begin{equation}\label{ae7}
\frac{n-1}{\sqrt{e}}+1<b_n<\frac{n-\frac{3}{2}}{\sqrt{e}}+\frac{5}{2}.
\end{equation}
In particular,
\begin{equation}\label{ae8}
\frac{n+5}{2}<b_n<\frac{2n-1}{3}.
\end{equation}
\end{lem}
\begin{proof}
By (\ref{v4e1}), we have
\begin{equation}\label{ae4}
\frac{1}{2}<\sum_{i=b_n}^{n}\frac{1}{i-2}=\sum_{i=b_n-2}^{n-2}\frac{1}{i}<\int_{b_n-\frac{5}{2}}^{n-\frac{3}{2}}\frac{dx}{x}=\log\left(\frac{n-\frac{3}{2}}{b_n-\frac{5}{2}}\right)
\end{equation}
and
\begin{equation}\label{ae5}
\frac{1}{2}\ge\sum_{i=b_n+1}^{n}\frac{1}{i-2}=\sum_{i=b_n-1}^{n-2}\frac{1}{i}>\int_{b_n-1}^{n-1}\frac{dx}{x}=\log\left(\frac{n-1}{b_n-1}\right).
\end{equation}
By (\ref{ae4}), we have $b_n<\frac{n-\frac{3}{2}}{\sqrt{e}}+\frac{5}{2}$; and from (\ref{ae5}), $b_n>\frac{n-1}{\sqrt{e}}+1$, establishing (\ref{ae7}). Since $\frac{n-\frac{3}{2}}{\sqrt{e}}+\frac{5}{2}<\frac{2n-1}{3}$ and $\frac{n-1}{\sqrt{e}}+1>\frac{n+5}{2}$ (for $n\ge32$), we have $\frac{n+5}{2}<b_n<\frac{2n-1}{3}$. The proof is complete.
\end{proof}

From (\ref{v4e1}) and (\ref{v4e2}), we have
\begin{align*}
(b_n-2)(n-2)&=\frac{(b_n-2)(2n-4)}{2}\\
&<u_n=2n-4+(b_n-2)(2n-4)\sum_{i=b_n+1}^n\frac{1}{i-2}\\
&\le 2n-4+\frac{(b_n-2)(2n-4)}{2}=b_n(n-2),
\end{align*}
i.e.
\begin{equation}\label{ae6}
(b_n-2)(n-2)<u_n\le b_n(n-2).
\end{equation}

\begin{rem}
The assumption of $n\ge32$ is needed for Lemmas \ref{sl4.2}--\ref{sl4.7} since the following proofs of the lemmas rely on $(\ref{ae8})$.
\end{rem}


\begin{proof}[\bf Proof of Lemma \ref{sl4.2}]
(i) Note ({\it cf.} Remark 2.1) that  $a_n=\lceil x_0\rceil<x_0+1$ where $x_0$ is the smaller root of $f_n(x)=0$. We now  show $f_n(b_n-1)<0$ (which implies that  $a_n<x_0+1<(b_n-1)+1=b_n$). We have
\begin{align*}
f_n(b_n-1)&=3(b_n-1)^2-(1+4n)(b_n-1)+(n-2)b_n+2(n+1)+u_n\\
&\le3(b_n-1)^2-(1+4n)(b_n-1)+(n-2)b_n+2(n+1)+b_n(n-2)\;\;\mbox{(by (\ref{ae6}))}\\
&=(b_n-3)\left[3b_n-(2n+2)\right]<0\;\;\mbox{(by (\ref{ae8})).}
\end{align*}
This proves (i).

(ii) Note that
\begin{align*}
f_n(b_n)&\le3b_n^2-(1+4n)b_n+(n-2)b_n+2(n+1)+b_n(n-2)\\
&=(b_n-1)\left[3b_n-(2n+2)\right]<0\;\;\mbox{(by (\ref{ae8})).}
\end{align*}
This proves that $b_n<y_0$.

(iii) By (\ref{ae8}) and (ii), $y_0>b_n>\frac{n+5}{2}>\frac{n+4}{3}$. We now show  $f_n\left(\frac{n+4}{3}\right)>0$
(which implies that $\frac{n+4}{3}<x_0\le \lceil x_0 \rceil =a_n$). By (\ref{ae6}),
\begin{align*}
f_n\left(\frac{n+4}{3}\right)&=-n^2-3n+4+(n-2)b_n+2(n+1)+u_n\\
&>-n^2-3n+4+(n-2)b_n+2(n+1)+(b_n-2)(n-2)\;\;\mbox{(by (\ref{ae6}))}\\
&=(n-2)\left(2b_n-(n+5)\right)>0\;\;\mbox{(by (\ref{ae8})).}
\end{align*}
The proof is complete.
\end{proof}

\begin{proof}[\bf Proof of Lemma \ref{sl4.4}]
By Lemma \ref{sl4.2}, $a_n<b_n$. (i) Let
\begin{align*}
Q_{i}&=\{X_\ell\neq2\;\mbox{for}\;a_n\le\ell\le i-1,X_i=2\},\; a_n\le i\le b_n-1;\\
Q'_{i}&=\{X_\ell\neq2\;\mbox{for}\;a_n\le\ell\le b_n-1,X_\ell\neq2,3\;\mbox{for}\;b_n\le\ell\le i-1,X_i=2\},\;i\ge b_n;\\
\mbox{and}\;\;Q''_{i}&=\{X_\ell\neq2\;\mbox{for}\;a_n\le\ell\le b_n-1,X_\ell\neq2,3\;\mbox{for}\;b_n\le\ell\le i-1,X_i=3\},\;i\ge b_n.
\end{align*}
Since $X_\ell$ is uniformly distributed over $\{1,2,\dots,\ell\}$, the $X_\ell's$ are independent  and
$R_i$ is conditionally independent of $X_1,\dots,X_{i-1}$ given $X_i$, we have
\begin{align*}
P(Q_{i})=\frac{(a_n-1)}{i(i-1)},\; &P(R_i=3|Q_{i})=y_i(2)\;\;\mbox{for}\;\;a_n\le i\le b_n-1, \\ P(Q'_{i})=P(Q''_{i})=\frac{(a_n-1)(b_n-2)}{i(i-1)(i-2)},\; &P(R_i=3|Q'_{i})=y_i(2),\; P(R_i=3|Q''_{i})=y_i(3),\;\;\mbox{for}\;\;i\ge b_n.
\end{align*}
Thus, by (\ref{e15}) and (\ref{v4e2}), for $j<a_n$,
\begin{align}
h_j(x_j)&=\sum_{i=a_n}^{n}P(R_i=3\;\mbox{and the $i$-th candidate is selected under}\;\tau_{3,n})\notag\\
&=\sum_{i=a_n}^{b_n-1}P(Q_{i})P(R_i=3|Q_{i})+\sum_{i=b_n}^{n}\Big[P(Q'_{i})P(R_i=3|Q'_{i})+P(Q''_{i})P(R_i=3|Q''_{i})\Big]\notag\\
&=\sum_{i=a_n}^{b_n-1}\frac{(a_n-1)}{i(i-1)}y_{i}(2)+\sum_{i=b_n}^{n}\left[\frac{(a_n-1)(b_n-2)}{i(i-1)(i-2)}\left(y_i(2)+y_i(3)\right)\right]\notag\\
&=\frac{a_n-1}{n(n-1)(n-2)}\left[\sum_{i=a_n}^{b_n-1}2(n-i)+(b_n-2)\sum_{i=b_n}^{n}\frac{2n-i-2}{i-2}\right]\notag\\
&=\frac{a_n-1}{n(n-1)(n-2)}\Bigg[(2n-a_n-b_n+1)(b_n-a_n)-(b_n-2)(n-b_n+1)\notag\\
&\hspace*{4cm}\left.+(b_n-2)(2n-4)\sum_{i=b_n}^{n}\frac{1}{i-2}\right]\notag\\
&=\frac{(a_n-1)\left[a_n^2-(1+2n)a_n+(n-2)b_n+2(n+1)+u_n\right]}{n(n-1)(n-2)}=:c_n\label{ae9}.
\end{align}
This proves (i) for $j<a_n$. The other cases can be treated similarly.

(ii) By (i), for $j<a_n-1$, $h_{j+1}(i)$ does not depend on $i$, so that $\frac{1}{j+1}\sum_{i=1}^{j+1}h_{j+1}(i)=c_n$.
To establish the identity for $j=a_n-1$, we have by (i) that $h_{a_n}(2)=y_{a_n}(2)$ and
$$
h_{a_n}(i)=\frac{a_n\left(a_n^2+(1-2n)a_n+(n-2)b_n+2+u_n\right)}{n(n-1)(n-2)}\;\;\mbox{for $i\neq2$ with $1\le i\le a_n$}.
$$
So,
\begin{align*}
\frac{1}{a_n}\sum_{i=1}^{a_n}h_{a_n}(i)&=\frac{1}{a_n}\left\{y_{a_n}(2)+(a_n-1)\left[\frac{a_n\left(a_n^2+(1-2n)a_n+(n-2)b_n+2+u_n\right)}{n(n-1)(n-2)}\right]\right\}\\
&=\frac{1}{a_n}\left\{\frac{2a_n(a_n-1)(n-a_n)}{n(n-1)(n-2)}+(a_n-1)\left[\frac{a_n\left(a_n^2+(1-2n)a_n+(n-2)b_n+2+u_n\right)}{n(n-1)(n-2)}\right]\right\}\\
&=\frac{(a_n-1)\left[a_n^2-(1+2n)a_n+(n-2)b_n+2(n+1)+u_n\right]}{n(n-1)(n-2)}=c_n.
\end{align*}
This proves (ii) for the case $j<a_n$. The other cases can be treated similarly.
\end{proof}

\begin{proof}[\bf Proof of Lemma \ref{sl4.5}]
Since, by Lemma \ref{sl4.4}(ii), $\frac{1}{j+1}\sum_{i=1}^{j+1}h_{j+1}(i)=c_n$ for $j<a_n$ where $c_n$ is defined in (\ref{ae9}), we need to show \begin{equation}\label{se4.4}
\max\{y_j(i):i=1,2,3,\;j<a_n\}<c_n,
\end{equation}
where $y_j(i)$ is given in (\ref{e15}).
Since $y_j(2)>y_j(3)$ if and only if $2(n-j)>j-2$ (i.e. $j<\frac{2n+2}{3}$) and, since by Lemma \ref{sl4.2}(i) and (\ref{ae8}), $a_n<b_n<\frac{2n-1}{3}$, we have $y_j(2)>y_j(3)$ for $j<a_n$, implying that
\begin{equation}\label{se4.5}
\max_{j<a_n}y_j(2)>\max_{j<a_n}y_j(3).
\end{equation}

Noting that $y_j(1)\ge y_{j+1}(1)$ if and only if $j\ge\frac{n-2}{3}$, we have
\begin{equation*}
\max_{1\le j\le n}y_j(1)=y_{\lceil\frac{n-2}{3}\rceil}(1)\le y_{\lceil\frac{n-2}{3}\rceil+1}(2),
\end{equation*}
where the inequality is due to the fact that $y_j(1)\le y_{j+1}(2)$ for $j\ge(n-2)/3$.
By Lemma \ref{sl4.2}(iii), $a_n>\frac{n+4}{3}>\lceil\frac{n-2}{3}\rceil+1$. So,
\begin{equation}\label{se4.6}
\max_{1\le j\le n}y_j(1)=y_{\lceil\frac{n-2}{3}\rceil}(1)\le y_{\lceil\frac{n-2}{3}\rceil+1}(2)\le\max_{j<a_n}y_j(2).
\end{equation}
Moreover, $y_j(2)\le y_{j+1}(2)$ if and only if $j\le\lfloor\frac{2n-1}{3}\rfloor$, which together with $a_n<\frac{2n-1}{3}$ implies that
\begin{equation}\label{se4.7}
\max_{j<a_n}y_j(2)=y_{a_n-1}(2).
\end{equation}
In view of (\ref{se4.5}), (\ref{se4.6}) and (\ref{se4.7}), (\ref{se4.4}) holds if we can show that
\begin{equation*}
y_{a_n-1}(2)<c_n,
\end{equation*}
i.e.
$$
3a_n^2-(4n+7)a_n+(n-2)b_n+6(n+1)+u_n>0,
$$
which is equivalent to $f_n(a_n-1)>0$. This holds by (\ref{v4e4}). The proof is complete.
\end{proof}

\begin{proof}[\bf Proof of Lemma \ref{sl4.6}]
(i) Note that
\begin{align*}
\frac{n(n-1)(n-2)}{j}\left[y_j(2)-\frac{1}{j+1}\sum_{i=1}^{j+1}h_{j+1}(i)\right]&=2(j-1)(n-j)-j^2-(1-2n)j-(n-2)b_n-2-u_n\\
&=-3j^2+(1+4n)j-(n-2)b_n-2(n+1)-u_n\\
&=-f_n(j)\ge0,
\end{align*}
where the inequality holds since $f_n(j)\le0$ for $x_0\le a_n\le j<b_n<y_0$ where $x_0$ and $y_0$ denote the two roots of $f_n(x)=0$.

(ii) Note that
\begin{align*}
&\frac{n(n-1)(n-2)}{j}\left[y_j(1)-\frac{1}{j+1}\sum_{i=1}^{j+1}h_{j+1}(i)\right]\\
=&(n-j-1)(n-j)-j^2-(1-2n)j-(n-2)b_n-2-u_n\\
=&n^2-n-(n-2)b_n-2-u_n\\
<&n^2-n-(n-2)b_n-2-(b_n-2)(n-2)\;\;\mbox{(by (\ref{ae6}))}\\
=&(n-2)(n+3-2b_n)<0\;\;\mbox{(by (\ref{ae8}))}.
\end{align*}
This proves (ii).

(iii) Note that
\begin{align*}
\frac{n(n-1)(n-2)}{j}\left[y_j(3)-\frac{1}{j+1}\sum_{i=1}^{j+1}h_{j+1}(i)\right]&=(j-1)(j-2)-j^2-(1-2n)j-(n-2)b_n-2-u_n\\
&=(n-2)(2j-b_n)-u_n\\
&<(n-2)(2j-b_n)-(b_n-2)(n-2)\;\mbox{(by (\ref{ae6}))}\\
&=2(n-2)(j+1-b_n)\le0,
\end{align*}
where the last inequality follows since $j\le b_n-1$. The proof is complete.
\end{proof}

\begin{proof}[\bf Proof of Lemma \ref{sl4.7}]
We claim that
\begin{align}
&\frac{j-1}{n-j}\sum_{i=j+1}^{n}\frac{1}{i-2}\;\;\mbox{is increasing in}\;\;2\le j<n;\label{ae10}\\
\mbox{and}\;\;&\frac{1}{n-j}\sum_{i=j+1}^{n}\frac{1}{i-2}\;\;\mbox{is decreasing in}\;\;2\le j<n.\label{ae11}
\end{align}
Note that for $j=2,\dots, n-2$,
\begin{align*}
\frac{j-1}{n-j}\sum_{i=j+1}^{n}\frac{1}{i-2}-\frac{j}{n-j-1}\sum_{i=j+2}^{n}\frac{1}{i-2}&=\frac{1}{n-j}-\frac{n-1}{(n-j)(n-j-1)}\sum_{i=j+2}^{n}\frac{1}{i-2}\\
&=\frac{n-1}{n-j}\left(\frac{1}{n-1}-\frac{1}{n-j-1}\sum_{i=j+2}^{n}\frac{1}{i-2}\right)<0,
\end{align*}
establishing  (\ref{ae10}). A similar argument yields (\ref{ae11}).

(i) By (\ref{e15}) and Lemma \ref{sl4.4}(ii), for $b_n\le j\le n-1$,
\begin{align*}
&\frac{n(n-1)(n-2)}{j}\left[y_j(1)-\frac{1}{j+1}\sum_{i=1}^{j+1}h_{j+1}(i)\right]\\
=&(n-j-1)(n-j)-(j-1)\left[(2n-4)\sum_{i=j+1}^{n}\frac{1}{i-2}-(n-j)\right]\\
=&(n-j)(n-2)\left[1-\frac{2(j-1)}{n-j}\sum_{i=j+1}^{n}\frac{1}{i-2}\right]\\
\le&(n-j)\left[1-\frac{2(b_n-1)}{n-b_n}\sum_{i=b_n+1}^{n}\frac{1}{i-2}\right]\;\;\mbox{(by (\ref{ae10}))}\\
<&(n-j)\left[1-\frac{2(b_n-1)}{n-2}\right]\\
<&0\;\;\mbox{(since $b_n>\frac{n+5}{2}$ by (\ref{ae8}))}.
\end{align*}
This proves (i).

(ii) By (\ref{e15}) and Lemma \ref{sl4.4}(ii), for $b_n\le j\le n-1$,
\begin{align*}
\frac{n(n-1)(n-2)}{j(j-1)}\left[y_j(2)-\frac{1}{j+1}\sum_{i=1}^{j+1}h_{j+1}(i)\right]&=3(n-j)-(2n-4)\sum_{i=j+1}^{n}\frac{1}{i-2}\\
&=(n-j)\left[3-\frac{2n-4}{n-j}\sum_{i=j+1}^{n}\frac{1}{i-2}\right]\\
&\ge(n-j)\left[3-\frac{2n-4}{n-b_n}\sum_{i=b_n+1}^{n}\frac{1}{i-2}\right]\;\;\mbox{(by (\ref{ae11}))}\\
&\ge(n-j)\left[3-\frac{n-2}{n-b_n}\right]\;\mbox{(by (\ref{v4e1}))}\\
&>0\;\;\mbox{(since $b_n<(2n-1)/3$ by (\ref{ae8}))}.
\end{align*}
This proves (ii).

(iii) By (\ref{e15}) and Lemma \ref{sl4.4}(ii), for $b_n\le j\le n-1$,
\begin{align*}
\frac{n(n-1)(n-2)}{j(j-1)}\left[y_j(3)-\frac{1}{j+1}\sum_{i=1}^{j+1}h_{j+1}(i)\right]&=n-2-(2n-4)\sum_{i=j+1}^{n}\frac{1}{i-2}\\
&=(n-2)\left[1-2\sum_{i=j+1}^{n}\frac{1}{i-2}\right]\\
&\ge(n-2)\left[1-2\sum_{i=b_n+1}^{n}\frac{1}{i-2}\right]\\
&\ge0\;\;\mbox{(by (\ref{v4e1}))}.
\end{align*}
The proof is complete.
\end{proof}

\begin{proof}[\bf Proof of (\ref{b5})--(\ref{b6})]
It follows immediately from Lemma \ref{sl4.1} that $d_2=1/\sqrt{e}$. Let $x_0$ be the smaller root of $f_n(x)=0$, i.e.
\begin{align}
x_0:&=\frac{(1+4n)-\sqrt{(1+4n)^{2}-12[(n-2)b_n+2(n+1)+u_n]}}{6}\notag\\
&=\frac{2[(n-2)b_n+2(n+1)+u_n]}{1+4n+\sqrt{(1+4n)^{2}-12[(n-2)b_n+2(n+1)+u_n]}}.\label{b7}
\end{align}
Since $\frac{b_n}{n}\to d_2=1/\sqrt{e}$ and $\sum_{i=b_n}^{n}\frac{1}{i-2}\to\int_{1/\sqrt{e}}^{1}\frac{dx}{x}=\frac{1}{2}$ as $n\to\infty$,
\begin{equation}\label{b8}
\frac{u_n}{n^{2}}=\frac{(b_n-2)(2n-4)}{n^{2}}\sum_{i=b_n}^{n}\frac{1}{i-2}\to d_2\;\;\mbox{as}\;\;n\to\infty.
\end{equation}
By (\ref{b7}), (\ref{b8}) and $a_n=\lceil x_0\rceil$, we have
\begin{equation*}
d_1=\lim_{n\to\infty}\frac{a_n}{n}=\lim_{n\to\infty}\frac{x_0}{n}=\frac{2d_2}{2+\sqrt{4-6d_2}}=\frac{2}{2\sqrt{e}+\sqrt{4e-6\sqrt{e}}},
\end{equation*}
proving (\ref{b5}). By Lemma \ref{sl4.4}(i),
$$
p(3,n)=h_1(1)=\frac{(a_n-1)[a^{2}_{n}-(1+2n)a_{n}+(n-2)b_n+2(n+1)+u_n]}{n(n-1)(n-2)},
$$
which together with (\ref{b6}) and (\ref{b8}) yields
\begin{equation*}
p(3,\infty)=\lim_{n\to\infty}p(3,n)=d_1(d^{2}_{1}-2d_{1}+2d_{2})=2d^{2}_{1}(1-d_{1})=\frac{8\left(2\sqrt{e}-2+\sqrt{4e-6\sqrt{e}}\right)}{\left(2\sqrt{e}+\sqrt{4e-6\sqrt{e}}\right)^{3}},
\end{equation*}
proving (\ref{b6}).
\end{proof}

\section{A computer program in Mathematica for verification of Theorem \ref{st4.1} for $\mathbf{3\le n\le 31}$}

\newcommand{\mathsym}[1]{{}}
\newcommand{\unicode}[1]{{}}
\newcounter{mathematicapage}
\noindent
\(\pmb{\text{Clear}[f,u,n,j,x];}\\
\pmb{\text{For}[n=3,n<32,n\text{++},}\\
\pmb{u[\text{n$\_$},\text{j$\_$},\text{x$\_$}]\text{:=}\text{Which}\left[x==1,\frac{(n-j+1)*(j-2)*(j-1)}{n*(n-1)*(n-2)},x==2,\frac{2*(n-j+1)*(n-j)*(j-1)}{n*(n-1)*(n-2)},\right.}\\
\pmb{\left.,x==3,\frac{(n-j+1)*(n-j)*(n-j-1)}{n*(n-1)*(n-2)},\text{True},0\right];}\\
\pmb{\text{For}[j=1,j\leq  n,j\text{++},}\\
\pmb{\text{For}[x=1,x\leq n,x\text{++},}\\
\pmb{f[n,j,x]=\text{If}\left[j>1,\text{Max}\left[u[n,j,x],\frac{1}{n-j+2}*\left(\sum _{i=1}^{n-j+2} f[n,j-1,i]\right)\right],\text{Which}[x==3,1,x\neq
3,0]\right]}\\
\pmb{]\;\;\text{(*This sets the values backwards*)}}\\
\pmb{]}\\
\pmb{]}\)\\[4mm]

\noindent\(\pmb{\text{Clear}[y,v,b,n];}\\
\pmb{y[\text{n$\_$},\text{j$\_$},\text{x$\_$}]\text{:=}u[n,n+1-j,x];\;\text{(*Define the conditional probability y*)}}\\
\pmb{v[\text{n$\_$},\text{j$\_$},\text{x$\_$}]\text{:=}f[n,n+1-j,x];\;\text{(*Define the value function*)}}\\
\pmb{b[3]=3;\;\text{(*Define the threshold $b_n$*)}}\\
\pmb{\text{For}[n=4,n<32,n\text{++},}\\
\pmb{\text{For}\left[i=2,i<n,i\text{++},\text{If}\left[\sum _{k=i+1}^n \frac{1}{k-2}\leq \frac{1}{2},\;i\;\;\&\&\;\;\text{Break}[]\right]\right];}\\
\pmb{b[n]=i}\\
\pmb{]}\)\\[4mm]

\noindent\(\pmb{\text{Clear}[a,n,j];}\\
\pmb{a[\text{n$\_$}]\text{:=}\text{Ceiling}\left[\frac{\left(1+4n-\sqrt{(1+4n)^2-12\left((n-2)b[n]+2(n+1)+(b[n]-2)(2n-4)\sum _{j=b[n]}^n \frac{1}{j-2}\right)}\right)}{6}\right];\text{(*Define
the threshold $a_n$*)}}\\
\pmb{}\\
\pmb{\text{For}[n=3,n<32,n\text{++},}\\
\pmb{\text{If}[a[n]-b[n]>0,\text{Print}[n]\;\&\&\; \text{Break}[]]\;\;\text{(*}\text{This verifies}\;\text{that}\; \;\text{$a_n$}<\text{$b_n$}\;\text{for}\;3\leq  n\leq  31\text{*)}}\\
\pmb{]}\)\\[4mm]

\newpage
\noindent\(\pmb{\text{Clear}[i,j,n,x];}\\
\pmb{\text{For}[n=3,n<32,n\text{++},}\\
\pmb{\text{For}[j=1,j<a[n],j\text{++},}\\
\pmb{\text{For}[x=1,x\leq  j,x\text{++},}\\
\pmb{\text{If}\left[y[n,j,x]\ge\frac{1}{j+1}*\sum _{i=1}^{j+1} v[n,j+1,i],\text{Print}[\{n,j,x\}]\;\&\&\; \text{Break}[]\right]}\\
\pmb{]}\\
\pmb{]}\\
\pmb{]\;\;\text{(*}\text{This verifies}\;\text{Lemma}\;2.3\; \text{for}\;3\leq  n\leq  31\text{*)}}\)\\[4mm]

\noindent\(\pmb{\text{Clear}[i,j,n];}\\
\pmb{\text{For}[n=3,n<32,n\text{++},}\\
\pmb{\text{For}[j=a[n],j<b[n],j\text{++},}\\
\pmb{\text{If}\left[y[n,j,2]<\frac{1}{j+1}*\sum _{i=1}^{j+1} v[n,j+1,i]\left\| y[n,j,1]\geq  \frac{1}{j+1}*\sum _{i=1}^{j+1} v[n,j+1,i]\right\|\right.}\\
\pmb{y[n,j,3]\geq\frac{1}{j+1}*\sum _{i=1}^{j+1} v[n,j+1,i],\text{Print}[\{n,j,x\}]\;\&\&\;\text{Break}[]]}\\
\pmb{]}\\
\pmb{]\;\text{(*}\text{This verifies}\;\text{Lemma}\;2.4 \;\text{for}\;3\leq  n\leq  31\text{*)}}\)\\[4mm]

\noindent\(\pmb{\text{Clear}[i,j,n];}\\
\pmb{\text{For}[n=3,n<32,n\text{++},}\\
\pmb{\text{For}[j=b[n],j<n,j\text{++},}\\
\pmb{\text{If}\left[y[n,j,1]\geq  \frac{1}{j+1}*\sum _{i=1}^{j+1} v[n,j+1,i]\left\| y[n,j,2]< \frac{1}{j+1}*\sum _{i=1}^{j+1} v[n,j+1,i] \right\|\right.}\\
\pmb{y[n,j,3]< \frac{1}{j+1}*\sum _{i=1}^{j+1} v[n,j+1,i],\text{Print}[\{n,j\}]\;\&\&\;\text{Break}[]]}\\
\pmb{]}\\
\pmb{]\;\text{(*}\text{This verifies}\;\text{Lemma}\;2.5\; \text{for}\;3\leq  n\leq  31\text{*)}}\)

\section*{Acknowledgements}
\hspace*{18pt}The authors gratefully acknowledge support from the Ministry of Science and Technology of Taiwan, ROC.





\begin{thebibliography}{99}
\footnotesize
\itemsep=2pt


\bibitem{ref3}
    {\sc Chow, Y.-S., Robbins, H. and Siegmund, D.} (1971).
    {\em Great Expectations: the Theory of Optimal Stopping}.
    Houghton Mifflin, Boston, MA.

\bibitem{ref4}
    {\sc Ferguson, T.S.} (1989).
    Who solved the secretary problem?
    {\em Statistical Science} {\bf 4}, 282--296.

\bibitem{ref5}
    {\sc Ferguson, T.S.}
    {\em Optimal Stopping and Applications.}
    Mathematics Department, UCLA.
    {\em http://www.math.ucla.edu/$\sim$tom/Stopping/Contents.html}.

\bibitem{ref6}
    {\sc Freeman, P. R.} (1983).
    The secretary problem and its extensions: a review.
    {\em Int. Statist. Rev.} {\bf 51}, 189--206.

\bibitem{ref0}
{\sc Graham, R., Knuth, D., and Patashnik, O.} (1994).
{\em Concrete Mathematics: A Foundation for Computer Science.}
Addison-Wesley Professional.




\bibitem{ref10}
    {\sc Lindley, D.V.} (1961).
    Dynamic programming and decision theory.
    {\em Appl. Statist.} {\bf 10}, 39--51.








\bibitem{ref16}
{\sc Rose, J.S.} (1982).
A problem of optimal choice and assignment.
{\em Oper. Res.} {\bf 30}, 172--181

     \bibitem{ref24}
     {\sc Rose, J.S.} (1982).
     Selection of nonextremal candidates from a sequence.
     {\em J. Optimization Theory Appl.} {\bf 38}, 207--219.



\bibitem{ref19}
    {\sc Samuels, S.M.} (1991).
    Secretary problems. In {\em Handbook of Sequential Analysis} (Statist. Textbooks Monogr. {\bf 118}), eds B. K. Ghosh and P.K. Sen, Marcel Dekker, New York, pp. 381--405.

    \bibitem{ref25}
    {\sc Suchwalko, A. and Szajowski, K.} (2002).
    Non standard, no information secretary problems.
    {\em Sci. Math. Jpn.} {\bf 56}, 443--456.

    \bibitem{ref26}
    {\sc Szajowski, K.} (1982).
    Optimal choice problem of $a$-th object.
    {\em Mat. Stos.} {\bf 19}, 51--65 (in Polish).


\bibitem{ref21}
{\sc Vanderbei, R.J.} (2012).
The postdoc variant of the secretary problem.
{\em Tech. Report.} {\em http://www.princeton.edu/$\sim$rvdb/tex/PostdocProblem/PostdocProb.pdf}


\end{thebibliography}
\end{document}